\documentclass[12pt,a4paper]{amsart}
\usepackage[top=35mm, bottom=32mm, left=30mm, right=30mm]{geometry}
\usepackage{mathptmx}
\usepackage{mathrsfs}
\usepackage{tikz}
\usepackage{verbatim}
\usepackage{url}
\usepackage[all]{xy}
\usepackage{xcolor}
\usepackage[colorlinks=true,citecolor=blue]{hyperref}
\usepackage{multicol}
\theoremstyle{plain}
\newtheorem{claim}{Claim}
\newtheorem{thm}{Theorem}[section]
\newtheorem{lem}[thm]{Lemma}
\newtheorem{prop}[thm]{Proposition}
\newtheorem{cor}[thm]{Corollary}

\theoremstyle{definition}
\newtheorem{defn}[thm]{Definition}
\newtheorem{exmp}[thm]{Example}
\newtheorem{rem}[thm]{Remark}
\newtheorem{ques}[thm]{Question}

\newcommand{\Z}{{\mathbb{Z}_+}}
\newcommand{\N}{\mathbb{N}}

\newcommand{\mZ}{\mathbb Z}

\newcommand{\ep}{\varepsilon}

\DeclareMathOperator{\diam}{diam}
\DeclareMathOperator{\supp}{supp}

\begin{document}

\title{Null systems in the non-minimal case}

\author[J.~Qiu]{Jiahao Qiu}
\address[J.~Qiu]{Wu Wen-Tsun Key Laboratory of Mathematics, USTC, Chinese Academy of Sciences and
School of Mathematics, University of Science and Technology of China,
Hefei, Anhui, 230026, P.R. China}
\email{qiujh@mail.ustc.edu.cn}

\author[J.~Zhao]{Jianjie Zhao}
\address[J.~Zhao]{Wu Wen-Tsun Key Laboratory of Mathematics, USTC, Chinese Academy of Sciences and
School of Mathematics, University of Science and Technology of China,
Hefei, Anhui, 230026, P.R. China}
\email{zjianjie@mail.ustc.edu.cn}

\date{\today}

\begin{abstract}
In this paper, it is shown that if a dynamical system is null and distal, then it is equicontinuous.
It turns out that a null system with closed proximal relation is mean equicontinuous.
As a direct application, it follows that a null dynamical system with dense minimal points is also mean equicontinuous.

Meanwhile, a distal system with trivial $\text{Ind}_{fip}$-pairs, and a non-trivial
regionally proximal relation of order $\infty$ is constructed.
\end{abstract}

\maketitle

\section{Introduction}
Let $X$ be a compact metric space with a metric $d$, and let $T$ be a continuous
map from $X$ to itself.
The pair $(X,T)$ will be called a \emph{(topological) dynamical system}.

\medskip

Entropy and related notions are the measurements of the complexity of a dynamical system.
Recall that a dynamical system $(X,T)$ is {\it null} if  the topological sequence entropy
is zero for any sequence. For a measurable dynamical system, the nullness is defined similarly.
Kushnirenko \cite{K} has shown that a measure-preserving
transformation $T$ has a discrete spectrum if and only if it is null.
Unfortunately, in the topological setting, there is no such theorem, since Goodman \cite{GM} obtained a minimal non-equicontinuous system which is null.
Under the minimality assumption, Huang, Li, Shao and  Ye in \cite{NS}
showed that a minimal null system looks like a minimal equicontinuous system. Precisely, if $(X,T)$ is a minimal null system,
$\pi:X\rightarrow X_{eq}$ is the factor map to the maximal equicontinuous factor,
then $(X,T)$ is uniquely ergodic, $\pi$ is almost 1-1 and is a measure-theoretic isomorphism. One purpose of the paper is
to show what a null system would be without the assumption of minimality.


\medskip
Localizing the notion of sequence entropy, the authors in \cite{NS} defined the notion of sequence entropy pairs and
showed that for a minimal distal system the set of sequence entropy
pairs coincides with the regionally proximal relation, and thus a minimal distal null system is equicontinuous.     
In this paper, we will show the result remains valid without the minimality assumption, i.e.
 a distal null system is equicontinuous (Theorem \ref{mainthm}). In fact we can show more, i.e. a distal, non-equicontinuous system
 has infinite sequence entropy (Theorem \ref{NOT}) which extends the previous result proved under the minimality assumption \cite{NS}.
We remark that the methods that we use to show the above two results are different from the ones in \cite{NS}.

\medskip
In \cite{LTY} the notion of mean equicontinuity was introduced, which is equivalent to the notion of {\it stable in the mean in the sense of Lyapunov}
given by Fomin in \cite{F51}.
A dynamical system $(X,T)$ is called \emph{mean equicontinuous} if for every $\ep>0$,
there exists a $\delta>0$ such that whenever $x,y\in X$ with $d(x,y)<\delta$,
\[\limsup_{n\to\infty}\frac{1}{n}\sum_{i=0}^{n-1}d(T^ix,T^iy)<\ep.\]
It was known that a minimal mean equicontinuous system is uniquely ergodic. The authors in \cite{LTY} proved that
if $(X,T)$ is a minimal mean equicontinuous system and $\pi:X\rightarrow X_{eq}$ is the factor map to the maximal equicontinuous factor,
then $\pi$ is a measure-theoretic isomorphism (it is not necessarily almost 1-1, see \cite{DG16}).
We refer to \cite{Felipe1, Felipe2, Felipe3, GLZ17, HL} for further study
on mean equicontinuity and related subjects. Just recently, the authors of the current paper in \cite{QZ18} showed
that many results (related to mean equicontinuity) proved before
under the minimality assumption hold for general systems.
In this paper, we will show that a null dynamical system
with a closed proximal relation is mean equicontinuous (Theorem \ref{equivalence}). A direct consequence is that
a null dynamical system with dense minimal points is mean equicontinuous. Note that it is easy to construct a null
dynamical system which is not mean equicontinuous. It is natural to ask if a null transitive system is mean equicontinuous.
Unfortunately, we can not answer the question in the current paper.
\medskip

To get an analogue result proved before in ergodic theory, in ~\cite{NI10} the notion of {\it regionally proximal
relation of order d} (denoted by $\mathbf{RP}^{[d]}$) was introduced by Host, Kra and Maass.
We refer to \cite{GGY,SY} for the further study on the topic.
In ~\cite{5p},  Dong, Donoso, Maass, Shao and Ye studied the properties of $\text{Ind}_{fip}$-pair.
They showed that
if $(X,T)$ is a minimal distal system and $(x,y)\in \mathbf{RP}^{[\infty]}=\cap_{d\geq1}\mathbf{RP}^{[d]}$,
 then $(x,y)\in \text{Ind}_{fip}(X,T)$.
Since our Theorem \ref{mainthm} can be restated as: if $(X,T)$ is distal, $\mathbf{RP}^{[1]}$ is non-trivial, then $\text{Ind}_{f}(X,T)$ is not trivial.
It is natural to conjecture that the result proved in \cite{5p} is valid without the minimality assumption. It is surprising for us that
it is not the case. That is, we construct a distal system such that $\mathbf{RP}^{[\infty]}$ is not trivial,
but ${\text{Ind}}_{fip}(X, T)$ is trivial (Example \ref{counter}).

\medskip

The paper is organized as follows.
In Section 2, we gather definitions and prove some initial theorems on the thickly syndetic
sets of $\mathbb{Z}_+^l$.
We show that a distal null system is equicontinuous in Section 3.
Moreover, it turns out that a distal but not equicontinuous system
has infinite sequence entropy.
In Section 4, we prove that
a null system with closed proximal relation is mean equicontinuous and give some applications.
In the final section, we construct
a distal system with trivial $\text{Ind}_{fip}$-pairs, and a non-trivial
regionally proximal relation of order $\infty$.

\medskip
\noindent {\bf Acknowledgments.}
The authors would like to thank Professor X. Ye and Dr. L. Xu for helping discussions and remarks.
We also thank Jian Li, Jie Li and T. Yu for useful suggestions.
The authors were supported by NNSF of China (11431012).

\section{Preliminaries}

In this section we recall some notions and aspects of the theories of topological dynamical systems.


\subsection{Cubes in $\mathbb{Z}_+^{l}$}
Denote by $\Z$ ($\N$, $\mZ$, respectively)
the set of all non-negative integers (positive integers, integers, respectively).



For $l\in \N$,
a \emph{cube} in $\mathbb{Z}_+^{l}$ is defined by
$$P_m^{(l)}=\{ \overline{p}=(p_{1}, p_{2}, \ldots, p_{l}): p_i\in \Z,0\leq p_{i} \leq m,i=1,\ldots,l\},$$
for some $m\in \N.$

A subset
$S^{(l)}=\{\overline{s}^{(k)}=(n^{(k)}_{1}, n^{(k)}_{2}, \ldots, n^{(k)}_{l}): k=1, 2, \ldots \} $
of $\mathbb{Z}_+^{l}$
is \emph{syndetic} if there is a cube
$P_m^{(l)}$,
such that for each element
$(a_{1}, a_{2}, \ldots, a_{l}) \in \mathbb{Z}_+^{l}$,
there is $(p_{1}, p_{2}, \ldots, p_{l}) \in P_m^{(l)}$ with
$(a_{1}+p_{1}, a_{2}+p_{2}, \ldots, a_{l}+p_{l}) \in S^{(l)}$.
We say that $m$ is a \emph{gap} of $S^{(l)}$.
Note that when $l=1$ we recover the classical definition of syndetic set in $\mathbb{Z}_+$.

A subset $T^{(l)}=\{\bar{n}^{(k)}=(n^{(k)}_{1}, n^{(k)}_{2}, \ldots, n^{(k)}_{l}): k=1, 2, \ldots  \} $
of $\mathbb{Z}_+^{l}$ is \emph{thickly syndetic} if for every cube
$P_n^{(l)}$,
there is a syndetic set $S_n^{(l)}=\{\bar{s}_n^{(k)}: k=1, 2, \ldots \}$
such that
$$  \{\bar{s}_n^{(k)}+P_n^{(l)} :k=1,2,\ldots \} \subset T^{(l)} .$$

Let $\mathcal{F}_{ts}^l$ be the set of all thickly syndetic sets in $\mathbb{Z}_+^{l}$.

\subsection{Compact metric spaces}
Let $(X,d)$ be a compact metric space.
For $x\in X$ and $\ep>0$, denote $B(x,\ep)=\{y\in X\colon d(x,y)<\ep\}$.
Denote by the product space $X\times X=\{(x,y)\colon x,y\in X\}$ and the diagonal $\Delta_X=\{(x,x)\colon x\in X\}$.
Let $U\subset X$, denote the diameter of the set $U$ by $\diam(U)=\sup\{d(x,y):x,y\in U\}$.

\subsection{Topological dynamics}
Let $(X,T)$ be a dynamical system.
The \emph{orbit} of a point $x\in X$, $\{x,Tx,T^2x, \ldots,\}$, is denoted by $Orb(x,T)$.
The set of limit points
of the orbit $Orb(x,T)$ is called the $\omega$-limit set of $x$, and is denoted by $\omega(x,T)$.

If $A$ is a non-empty closed subset of $X$ and $TA\subset A$, then $(A,T|_A)$ is called a \emph{subsystem} of $(X,T)$,
where $T|_A$ is the restriction of $T$ on $A$. If there is no ambiguity, we will use the notation $T$ instead of $T|_A$.

We say that a point $x\in X$ is \emph{recurrent} if $x\in\omega(x,T)$.
The system $(X,T)$ is called \emph{(topologically) transitive} if $\omega(x,T)=X$ for some $x\in X$, and
such a point $x$ is called a \emph{transitive point}.

The system $(X,T)$ is said to be \emph{minimal} if every point of $X$ is a transitive point.
A subset $Y$ of $X$ is called \emph{minimal} if $(Y,T)$ forms a minimal subsystem of $(X,T)$.
A point $x\in X$ is called \emph{minimal} if it is contained in a minimal set $Y$ or,
equivalently, if the subsystem $(\overline{Orb(x,T)},T)$ is minimal.

For $x\in X$ and $A\subset X$, let $N(x,A)=\{n\in\Z\colon T^nx\in A\}$.
If $U$ is a neighborhood of $x$, then the set $N(x,U)$ is called the set of \emph{return times} of the point $x$
to the neighborhood $U$.
The following result is well-known, see~\cite{F81} for example.
\begin{lem} \label{lem:rec-ip}
Let $(X,T)$ be a dynamical system and $x\in X$.
Then $x$ is minimal if and only if $N(x,U)$ is syndetic for every neighborhood $U$ of $x$.
\end{lem}

A pair of points $(x,y)\in X\times X$ is said to be \emph{proximal}
if for any $\ep>0$, there exists a positive integer $n$ such that $d(T^nx,T^ny)<\ep$.
Let $P(X,T)$ denote the collection of all proximal pairs in $(X,T)$.

The system $(X,T)$ is called \emph{distal} if the proximal relation is trivial
i.e., $P(X,T)=\Delta_X$.

The following proposition summarizes some basic properties of distal system:

\begin{prop} \cite[1, Chapters 5 and 7]{AM}\label{distalproperty}

\begin{enumerate}
  \item The Cartesian product of a finite family of distal systems is a distal system.
  \item If $(X,T)$ is a distal system and $Y$ is a closed and invariant subset of $X$,
  then $(Y,T)$ is a distal system.
\end{enumerate}
\end{prop}


Recall that a pair of points $(x,y)$ is called \emph{regionally proximal}
if for every $\ep>0$, there exist two points  $x',y'\in X$
with $d(x,x')<\ep$ and $d(y,y')<\ep$, and a positive integer $n$ such that $d(T^nx',T^ny')<\ep$.
Let $Q(X,T)$ be the set of all regionally proximal pairs in $(X,T)$.

Clearly, we have $P(X, T)\subset Q(X, T)$.

When $(X,T)$ and $(Y,S)$ are two dynamical systems and $\pi\colon X\to Y$ is a continuous onto map which
intertwines the actions (i.e., $\pi\circ T=S\circ \pi$),
one says that $(Y, S)$ is a \emph{factor} of $(X,T )$ or $(X, T)$ is an \emph{extension} of $(Y, S)$,
and $\pi$ is a \emph{factor map}.
If $\pi$ is a homeomorphism, then we say that $\pi$ is a \emph{conjugacy} and
that the dynamical systems $(X,T)$ and $(Y,S)$ are \emph{conjugate}.

If $\pi\colon (X,T)\to (Y,S)$ is a factor map, then $R_\pi=\{(x,x')\in X\times X\colon \pi(x)=\pi(x')\}$
is closed $T\times T$-invariant equivalence relation, that is $R_\pi$ is a closed subset of $X \times X$ and
if $(x,x')\in R_\pi$, then $(Tx,Tx')\in R_\pi$.
Conversely, if $R$ is a closed $T\times T$-invariant equivalence relation on $X$,
then the quotient space $X/R$ is a compact metric space
and $T$ naturally induces an action on $X/R$ by $T_R([x])=[Tx]$.
Then $(X/R,T_R)$ forms a dynamical system and the quotient map $\pi_R\colon X\to X/R$ is a factor map.
Hence there is a one-to-one correspondence between factors and closed invariant equivalence relations,
we will use them interchangeably.
\subsection{Invariant measures}
Let $(X,T)$ be a dynamical system and $M(X,T)$ be the set of $T$-invariant regular Borel probability measures on $X$.
An invariant measure is ergodic if and only if it is an extreme point of $M(X,T)$.
Denote by $M^e(X,T)$ the set of all ergodic measure.

For $\mu \in M(X, T)$. We define the \emph{support} of $\mu$ by
$\supp(\mu)=\{x\in X: \mu(U)>0 \;\text{for any}$ neighborhood $U$ of $x \}.$
The \emph{support} of a dynamical system $(X,T)$, denoted by $\supp(X,T)$,
is the smallest closed subset $C$ of $X$ such that $\mu(C)=1$ for all $\mu\in M(X,T)$.

Let $(X,T)$ be a dynamical system.
We call $(X,T)$ an \emph{$E$-system} if it is transitive and there exists $\mu\in M(X,T)$ such that $\supp(\mu)=X$.
We say that $(X,T)$ is \emph{uniquely ergodic} if $M(X,T)$ consists of a single measure.

\subsection{Topological sequence entropy}

Let $\mathrm{A}=\{0\leq t_1< t_2<\cdots\}\subseteq \Z$
be an increasing sequence of natural numbers and $\mathcal{U}$
be a finite cover of $X$. The \emph{topological sequence entropy}
of $\mathcal{U}$ with respect to $(X,T)$ along $\mathrm{A}$ is
defined by
$$
h_{top}^{\mathrm{A}}(T,\mathcal{U})=\limsup_{n\rightarrow \infty}\frac{1}{n}\log \mathrm{N}(\bigvee_{i=1}^{n}T^{-t_i}\mathcal{U})
$$
where $\mathrm{N}(\bigvee_{i=1}^{n}T^{-t_i}\mathcal{U})$
is the minimal cardinality among all cardinalities of sub-covers of $\bigvee_{i=1}^{n}T^{-t_i}\mathcal{U}$.
The \emph{topological sequence entropy} of $(X,T)$ along $\mathrm{A}$ is
$$h_{top}^{\mathrm{A}}(T)=\sup_{\mathcal{U}}h_{top}^{\mathrm{A}}(T,\mathcal{U}),$$
where supremum is taken over all finite open covers of $X$ (that is, made up of open sets).
If $\mathrm{A}=\mathbb{N}$ we recover standard topological entropy.
In this case we omit the superscript $\mathbb{N}$.

\subsection{Sequence entropy $n$-tuple}
Given a dynamical system $(X,T)$ and an integer $n\geq 2$, the $n$-th product system is the dynamical system
$(X^{(n)},T^{(n)})$, where $X^{(n)}$ is the cartesian product of $X$ with itself $n$
times and $T^{(n)}$ represents the simultaneous action $T$ in each coordinate of $X^{(n)}$.
The diagonal of $X^{(n)}$ is denoted by $\Delta_n(X)=\{(x,\ldots,x)\in X^{(n)} :x\in X\}.$

Let $(x_i)_{i=1}^n\in X^{(n)}$. A finite cover $\mathcal{U}=\{U_1,\ldots,U_k\}$ of $X$
is  said to be \emph{admissible cover } with respect to $(x_i)_{i=1}^n$ if for each
$1\leq j\leq k$ there exists $1\leq i_j\leq n$ such that $x_{i_j}$ is not contained in the
closure of $U_j$.

\begin{defn}
Let $(X,T)$ be a dynamical system. An $n$-tuple $(x_i)_{i=1}^n\in X^{(n)},n \geq 2$, is called
a sequence entropy $n$-tuple if for some $1\leq i,j \leq n,x_i \neq x_j,$
for any admissible open cover $\mathcal{U}$ with respect to $(x_i)_{i=1}^n$
there exists an increasing sequence of natural numbers $\mathrm{A}$
such that $h_{top}^{\mathrm{A}}(T,\mathcal{U})>0.$
\end{defn}

We denote by $SE_n(X,T)$ the set of sequence entropy $n$-tuples.
Denote by $SE_n^e(X,T)$ the set of essential sequence entropy $n$-tuples,
which means $(x_1,\ldots,x_n)\in SE_n(X,T)$ with $x_i\neq x_j,1\leq i<j\leq n$.
Sequence entropy 2-tuples are called sequence entropy pairs.
Denote by $SE(X, T)$ the set of all sequence entropy pairs.

\begin{rem}\label{AA}
By the definition, it follows that
if $(x_1,\ldots,x_n)\in SE_n(X,T)$,
then we have $(x_{i_1},\ldots,x_{i_k})\in SE_k(X,T)\cup\Delta_k(X)$
for any $1\leq i_1<\cdots<i_k\leq n,k\geq2$.
\end{rem}

We have the following proposition:

\begin{prop}\cite[Proposition 3.2]{HMY}\label{SE}
Let $(X,T)$ be a dynamical system.
\begin{enumerate}
  \item If $\mathcal{U}=\{U_1,\ldots,U_k\}$ is an open cover of $X$ with $h_{top}^{\mathrm{A}}(T,\mathcal{U})>0$
  for some increasing sequence of natural numbers $\mathrm{A}$, then for all $1\leq i\leq n$
  there exists $x_i \in U_i^c$ such that $(x_i)_{i=1}^n$ is a sequence entropy $n$-tuple.
  \item $SE_{n}(X, T)\cup \Delta_n(X)$ is a closed $T^{(n)}$-invariant subset of $X^{(n)}$.

  \item Let $\pi :(Y,S) \longrightarrow (X,T)$ be a factor map of dynamical system.

  \begin{enumerate}
   \item If $(x_i)_{i=1}^n\in SE_n(X,T)$, then for all $1\leq i\leq n$
   there exists $y_i\in Y$ such that $\pi(y_i)=x_i$ and $(y_i)_{i=1}^n\in SE_n(Y,S).$
  Moreover, if $SE_n^e(X,T)\neq \emptyset$, then we have $SE_n^e(Y,S)\neq \emptyset$.
   \item If $(y_i)_{i=1}^n\in SE_n(Y,S)$ and $(\pi(y_i))_{i=1}^n \notin \Delta_n(X),$
   then $(\pi(y_i))_{i=1}^n \in SE_n(X,T)$.
 \end{enumerate}
\end{enumerate}
\end{prop}

\begin{prop}\cite[Theorem 5.9]{IT07}\label{SES}
Let $(X,T)$ be a dynamical system and $(x_i)_{i=1}^n \in X^{(n)}$
with $x_i\neq x_j,i\neq j.$
If for any neighborhoods $U_i$ of $x_i$ ($i=1,\ldots,n$) respectively,
and any $k\in \N$ there is a sequence
$\mathrm{A}=\{0\leq t_1<t_2<\cdots<t_k\}$ of $\Z$
such that for any $s=(s(1),\ldots,s(k))\in \{1,\ldots,n\}^k$,
$$
\bigcap_{i=1}^k T^{-t_i}U_{s(i)}\neq \emptyset,
$$
then $(x_i)_{i=1}^n\in SE_n(X,T)$.
\end{prop}

\begin{rem}
  By this proposition, we know that a system $(X, T)$ is null if and only if $SE(X, T)= \emptyset$.
\end{rem}

\subsection{Independence and $\mathbf{RP}^{[d]}(X)$}
Let $(X,T)$ be a dynamical system and $\mathcal{A}=(A_{1}, \ldots, A_{k})$ be
a tuple of subsets of $X$. A subset $F\subset \mathbb{Z}_{+}$ is an $\emph{independence set}$ for $\mathcal{A}$ if for any nonempty finite subset
$J\subset F$ and any $s=(s(j):j\in J)\in \{1, 2, \ldots, k\}^{|J|}$ we have $\cap_{j\in J}T^{-j}A_{s(j)}$
$\neq \emptyset$.
For a finite subset $\{p_{1}, \ldots, p_{m}\}$ of $\mathbb{N}$, the $\emph{finite IP-set}$ generated by $\{p_{1}, \cdots, p_{m}\}$
is the set $\{\epsilon_{1}p_{1}+ \cdots+ \epsilon_{m}p_{m}: \epsilon_{i}\in \{0,1\}, 1 \leq  i \leq m\} \backslash \{0\}$.
A pair of points $(x,y)\in X\times X$ is an
$\text{Ind}_{fip}$-pair if and only if each $\mathcal{A}=(A_{1},A_{2})$, with $A_{1}$ and ${A_{2}}$
neighborhoods of $x$ and $y$ respectively, contains finite IP-independence sets
of length $m$ for any $m\in \N$.

A pair $(x,y)\in X\times X$ is said to be $\emph{regionally proximal of order d}$ if for any $\delta >0$,
there exist $x', y' \in X$ and a vector ${\bf{n}}=(n_{1}, \ldots, n_{d})\in \mathbb{Z}_+^{d}$ such that
$d(x, x')< \delta$, $d(y, y')< \delta$, and
$d(T^{{\bf{n}} \cdot  \epsilon} x', T^{{\bf{n}} \cdot \epsilon }y')< \delta $ for any
$\epsilon =(\epsilon_1,\ldots,\epsilon_d)\in \{0, 1\}^{d} \backslash \{\mathbf{0}\}$,
where ${\bf{n}} \cdot \epsilon=\sum_{i=1}^{d}\epsilon_{i}n_{i}$.
The set of regionally proximal pairs of order $d$ is denoted by $\mathbf{RP}^{[d]}(X)$.
In ~\cite{NI10} the authors showed that if the system is minimal and distal then $\mathbf{RP}^{[d]}(X)$ is an equivalence relation and
$(X/\mathbf{RP}^{[d]}(X), T)$ is the maximal $d$-step nilfactor of the system.

\subsection{$\mathcal{F}_{ts}^l$ is a filter}

We recall some notions related to a family. Denote by $\mathcal{P}^l=\mathcal{P}^l(\mathbb{Z}_+^l)$ the collection of all subsets of $\mathbb{Z}_+^l$.
A subset $\mathcal{F}^l$ of $\mathcal{P}^l$ is called a \emph{family}, if it is hereditary upward, i.e.,
$F_1 \subset F_2$ and $F_1 \in \mathcal{F}^l$ imply $F_2 \in \mathcal{F}^l$.
A family $\mathcal{F}^l$ is called \emph{proper} if it is a non-empty proper subset of $\mathcal{P}^l$,
i.e., neither empty nor all of $\mathcal{P}^l$.
A proper family $\mathcal{F}^l$ is called a \emph{filter} if $F_1, F_2 \in \mathcal{F}^l$ implies
$F_1 \cap F_2 \in \mathcal{F}^l$.
Recall that $\mathcal{F}_{ts}^l$ the set of all thickly syndetic sets in $\mathbb{Z}_+^{l}$.
Then we have the following theorem.

\begin{prop}\label{filter}
For $l\in\N$, if $F_1,F_2\in \mathcal{F}_{ts}^l$, then we have $F_1\cap F_2\in \mathcal{F}_{ts}^l$.
\end{prop}

\begin{proof}
For $l\in \N$,
assume that $F_1,F_2\in \mathcal{F}_{ts}^l$.

Fix $n\in \N$, as $F_1\in \mathcal{F}_{ts}^l$,
there is a syndetic set $S_n^{(l)}=\{\bar{s}_n^{(k)}: k=1, 2, \ldots \}$
with
$$  \{\bar{s}_n^{(k)}+P_n^{(l)}:k=1,2,\ldots  \} \subset F_1 .$$

Let $m\in  \N$ be a gap of  $S_n^{(l)}$.

Put $d=m+n$.
Again, as $F_2\in \mathcal{F}_{ts}^l$,
there is a syndetic set $S_d^{(l)}=\{\bar{s}_d^{(k)}: k=1, 2, \ldots \}$
with
$$  \{\bar{s}_d^{(k)}+P_d^{(l)} :k=1,2,\ldots \} \subset F_2 .$$

For every $k\in \N$,
there is a $\overline{m}_k\in P_m^{(l)}$ with $\bar{s}_d^{(k)}+\overline{m}_k\in S_n^{(l)},$
hence $\bar{s}_d^{(k)}+\overline{m}_k+P_n^{(l)}\subset F_1.$

Thus
\begin{align*}
F_1\cap F_2\supseteq &
 \{\bar{s}_n^{(k)}+P_n^{(l)}:k=1,2,\ldots  \}\cap
 \{\bar{s}_d^{(k)}+P_d^{(l)} :k=1,2,\ldots \}\\
&\supseteq  \{\bar{s}_d^{(k)}+\overline{m}_k+P_n^{(l)}
:\text{for some}\;\overline{m}_k\in P_m^{(l)},k=1,2,\ldots\}.
\end{align*}
Clearly, $\{\bar{s}_d^{(k)}+\overline{m}_k\}_{k=1}^{\infty}$ is syndetic
which implies $F_1\cap F_2\in \mathcal{F}_{ts}^l$.
\end{proof}

By induction, we can easily obtain the following corollary.

\begin{cor}\label{filter2}
For $l,n\in\N$, if $F_1,F_2,\ldots,F_n\in \mathcal{F}_{ts}^l$,
then $\cap_{i=1}^n F_i\in \mathcal{F}_{ts}^l$.
\end{cor}

\section{Systems with sequence entropy $n$-tuples}
In this section we will give the conditions under which a dynamical system has
a sequence entropy $n$-tuple.

\begin{lem}\label{semisimple}
Let $X$ be a dynamical system with fixed points $x_i,i=1,2,\ldots,n,\ n\geq2$.
If for any neighbourhood $W_i$ of $x_i,i=1,2,\cdots,n$, there exists a minimal point $y$
such that for every $i=1,2,\cdots,n$ there is $m_i\in \Z$ with $T^{m_i}y\in W_i$,
then
$(x_1,x_2,\ldots,x_n)\in SE_n^e(X,T)$.

\end{lem}

\begin{proof}

Let $U_i$ be a neighbourhood of $x_i,i=1,2,\cdots,n$.
There is a $\delta>0$ with $B(x_i,\delta)\subset U_i,i=1,2,\ldots,n$.
For a given $N\in \N$, by the uniform continuity, there is a $\delta'>0$ such that
$d(T^ju,T^jv)<\delta$ for every $1\leq j\leq N$ and $u,v\in X$ with $d(u,v)<\delta'$.

We can choose a minimal point $y\in X$ and $m_i\in\Z$ such that
$d(x_i,T^{m_i}y)<\frac{\delta'}{2}$ for all $i=1,2,\ldots,n$.
Put $V_i=B(x_i,\delta')$,
then $N(y,V_i)$ is not empty,
$ i=1,2,\ldots,n$.

For $l\in \N$ and $s \in\{1,2,\ldots,n\}^l$, let
$V^{(s)}=\prod_{i=1}^{l}V_{s(i)}$ and $U^{(s)}=\prod_{i=1}^{l}U_{s(i)}$. Then
$$
A_s^{(l)}=\{\bar{n}=(n_1,n_2,\ldots,n_l)\in \mathbb{Z}_+^l:(T^{n_1}y,\ldots,T^{n_l}y)
\in V^{(s)}\}=\prod_{i=1}^{l}N(y, V_{s(i)})$$
is syndetic since $y$ is a minimal point.

For every $\bar{n}=(n_1,n_2,\ldots,n_l)\in A_s^{(l)}$, $\bar{p}=(p_1,p_2,\ldots,p_l)\in P_N^{(l)}$
and $i=1,2,\ldots,l$,
we have $T^{n_i}y\in V_{s(i)}$.
It follows that $d(T^{n_i}y,x_{s(i)})<\delta'$
which implies $d(T^{n_i+p_i}y,x_{s(i)})<\delta$ for every $0\le p_i\le N$. Thus,

\begin{align*}
B_s^{(l)}
=&\{\bar{n}=(n_1,n_2,\ldots,n_l)\in\mathbb{Z}_+^l:\bigcap_{i=1}^{l}T^{-n_i}U_{s(i)}\neq\emptyset\}
\\&\supset\{\bar{n}=(n_1,n_2,\ldots,n_l)\in \mathbb{Z}_+^l:(T^{n_1}y,T^{n_2}y,\ldots,T^{n_l}y)\in U^{(s)} \}
\\&\supset \{\bar{n}+P_N^{(l)}:\bar{n}\in A_s^{(l)} \}.
\end{align*}
That is, $B_s^{(l)}$ is thickly syndetic.

By Corollary \ref{filter2},
it follows that
$$ \bigcap_{s \in\{1,\ldots,n\}^{l}}B_s^{(l)}\neq\emptyset.$$

Pick $(n_1,n_2,\ldots,n_l)   \in   \cap_{s \in\{1,\ldots,n\}^{l}}B_s^{(l)}$.
Without loss of generality, assume $n_{1}<n_2< \cdots <n_{l}$. Then,
$$\bigcap_{i=1}^{l}T^{-n_{i}}U_{s(i)}\neq\emptyset$$
for any $s \in\{1,2,\ldots,n\}^{l}$, which implies $(x_1,x_2,\ldots,x_n)\in SE_n(X,T)$
by Proposition ~\ref{SES}.
\end{proof}

\begin{thm}\label{mainthm}
Let $(X,T)$ be a distal and null system, then it is equicontinuous.
\end{thm}

\begin{proof}
Assume that $(X,T)$ is not equicontinuous,
then there is $(x,y)\in Q(X,T)\setminus \Delta_X$,
with $x_k,y_k,z\in X$ and $n_k\in \Z,k=1,2,\cdots$
such that
$$\lim_{k\rightarrow \infty}(x_k,y_k)=(x,y)\quad
\text{and}
\quad \lim_{k\rightarrow \infty}(T^{n_k}x_k,T^{n_k}y_k)=(z,z).$$

As $(X,T)$ is distal, $T$ is invertible and $(X\times X,T\times T)$ is pointwise minimal
by Proposition \ref{distalproperty}.
Without loss of generality, we may assume that $T^{-n_k}z\rightarrow w$,
then $((x,y),(w,w))\in Q(X\times X,T\times T)$ since
$$
\lim_{k\rightarrow \infty}d((T\times T)^{n_k}(x_k,y_k),(T\times T)^{n_k}(T^{-n_k}z,T^{-n_k}z))
=\lim_{k\rightarrow \infty}d((T^{n_k}x_k,T^{n_k}y_k),(z,z))=0.$$

Put $\mathbf{x}=(x,y),\mathbf{w}=(w,w),\mathbf{x_k}=(x_k,y_k),\mathbf{z}=(z,z),\widetilde{X}=X\times X$
and $\widetilde{T}=T\times T$.

Let
$$R=\{(\mathbf{u},\mathbf{v})\in \widetilde{X}\times \widetilde{X}:
\mathbf{u},\mathbf{v} \in \overline{Orb(\mathbf{x},\widetilde{T})}\}\cup \{(\mathbf{u},\mathbf{v})\in \widetilde{X}\times \widetilde{X}:
\mathbf{u},\mathbf{v} \in \overline{Orb(\mathbf{w},\widetilde{T})}\}\cup \Delta_{\widetilde{X}}.$$
It is easy to see that $R$ is
a $ \widetilde{T}\times \widetilde{T}$-invariant closed equivalence relation.
Let $Y=\widetilde{X} /R$ (i.e. we collapse $\overline{Orb(\mathbf{x},\widetilde{T})}$ to a point, and collapse $\overline{Orb(\mathbf{w},\widetilde{T})}$ to a point),
then $[\mathbf{x}],[\mathbf{w}]$ are fixed points in $Y$.

It is easy to see $\lim_{k\rightarrow\infty}d([\mathbf{x_k}],[\mathbf{x}])=0$
and $\lim_{k\rightarrow\infty}d(\widetilde{T}^{n_k}[\mathbf{x_k}],[\mathbf{z}])=0$
since $[\mathbf{w}]=[\mathbf{z}]$.
By Lemma \ref{semisimple}, it follows that $([\mathbf{x}],[\mathbf{z]})\in SE(Y,\widetilde{T})$
which implies $SE(\widetilde{X},\widetilde{T})\neq\emptyset$.
We can choose $((c_1,c_2),(d_1,d_2))\in SE(\widetilde{X},\widetilde{T})$
with $c_1\neq d_1$ by Proposition \ref{SE}.

Now consider the projection $\pi_1$ from $X\times X$ to the first coordinate,
again by Proposition \ref{SE}, we have $(c_1,d_1)\in SE(X,T).$
It is a contradiction which shows the theorem.
\end{proof}

By the proof above, we can show more. That is,
if $(X,T)$ is distal but not equicontinuous, 
then for every $n\geq 2,SE_n^e(X,T)\neq\emptyset.$

Before showing the statement, we need the following lemma.

\begin{lem}\label{SENN}
Let $(X, T)$ be a distal dynamical system with fixed points $x_1,x_2$.
If there exist points $y_k\in X$ and $n_k\in \Z,k=1,2,\cdots$ such that
\begin{equation}\label{condition}
\lim_{k\rightarrow\infty}d(y_k,x_1)=0\;\;
\text{and} \;\; \lim_{k\rightarrow\infty}d(T^{n_k}y_k,x_2)=0,
\end{equation}
then for every $n\geq 2$,
$SE_n^e(X,T)\neq \emptyset$.
\end{lem}

\begin{proof}
As $(X,T)$ is distal, $y_k, k=1, 2, \cdots$ are minimal.
By Lemma \ref{semisimple}, it is clear for $n=2$.
Let $U_1$ and $U_2$ be neighbourhoods of $x_1$ and $x_2$ respectively with $\overline{U_1}\cap \overline{U_2}=\emptyset$.

\begin{claim}
Let $U_1'=U_1\cap T^{-1}U_1$ and $U_2'=U_2\cap T^{-1}U_2.$
There is a $K\in\Z$
such that
for every $k\geq K$, we can choose $l_{k}\in\Z$
with $T^{l_{k}}y_{k}\notin U_1' \cup U_2 '$.
\end{claim}

\begin{proof}[Proof of the Claim]
Clearly, $U_1'$ and $U_2'$ are
still neighbourhoods of $x_1$ and $x_2$ respectively since $x_1$ and $x_2$ are fixed points.
As $\lim_{k\rightarrow\infty}d(y_k,x_1)=0$, there is a $K\in\Z$
such that $y_k\in U_1'$ for all $k\ge K$.

Fix $k\ge K,$
put $l_k=\min \{n\in\Z: T^{n+1}y_{k}\in U_2'\} ,$
then we have $T^{l_k}y_k\notin U_1'$.
Indeed, if $T^{l_k}y_k\in U_1'$, we have $T^{l_k}y_k,T^{l_k+1}y_k\in U_1$.
By the definition of $l_k$,
it follows that $T^{l_k+1}y_k\in U_2'\subseteq U_2$,
it is a contradiction.
Thus $T^{l_{k}}y_{k}\notin U_1' \cup U_2 '$.
\end{proof}

By considering the limit points of $\{T^{l_k}y_k\}_{k=K}^\infty$,
we may assume that $T^{l_k}y_k\rightarrow x_3$ as $k\rightarrow \infty$. Then, it is clear that
$x_3\neq x_1,x_2.$

Let $R_1=\{(u,v)\in X\times X: u,v\in \overline{Orb(x_3,T)}\}\cup \Delta_X$
and it is
a $T\times T$-invariant closed equivalence relation.
Put $X_1=X,X_2=X_1/R_1$, then we have $x_1,x_2,x_3$
(instead of $[x_1],[x_2],[x_3]$)
are fixed points of $X_2$ and $\lim_{k\rightarrow \infty}d(T^{l_k}y_k,x_3)=0$.
As $y_k$ is minimal in $X_1$, so is in $X_2$.

Again by Lemma \ref{semisimple}, we have $(x_1,x_2,x_3)\in SE_3(X_2,T)$
which implies $SE_3(X,T)=SE_3(X_1,T)\neq\emptyset,$
by Proposition \ref{SE}.
Moreover, we have $SE_3^e(X,T)\neq \emptyset$ since $x_3\neq x_1,x_2.$.

Now assume that $x_1,x_2,\ldots,x_m$ are fixed points of $X_{m-1}$
and $y_{k} \in X_{m-1}, n^{(1)}_k,n^{(2)}_k,\ldots,n^{(m)}_k\in \Z,n^{(1)}_k=0,k\in\N$
with $\lim_{k\rightarrow\infty}d(x_i,T^{n^{(i)}_k}y_k)=0,i=1,2,\ldots,m$.

Choose neighborhoods $U_{1}, U_{2}, \ldots, U_{m}$ of  $x_{1}, x_{2}, \ldots ,x_{m} $ respectively with
$\overline{U_{i}}\cap \overline{U_{j}}=\emptyset$ for $1\leq i \neq j \leq m$.

\begin{claim}
Let $U_i'=U_i\cap T^{-1}U_i,i=1,2,\ldots,m.$
There is a $K\in\Z$
such that
for every $k\geq K$, we can choose $m_{k}\in\Z$
with $T^{m_{k}}y_{k}\notin \cup_{i=1}^mU_i'$.
\end{claim}

\begin{proof}
It is similar to the proof of Claim 1.
\end{proof}

By considering the limit points of $\{T^{m_k}y_k\}_{k=K}^\infty$,
we may assume that $T^{m_k}y_k\rightarrow x_{m+1}$ as $k\rightarrow \infty$. It is clear that
$x_{m+1}\neq x_i,i=1,2,\ldots,m.$

Let $R_{m-1}=\{(u,v)\in X_{m-1}\times X_{m-1}: u,v\in \overline{Orb(x_{m+1},T)}\}\cup \Delta_{X_{m-1}}$.
Then it is a $T\times T$-invariant closed equivalence relation.
Put $X_{m}=X_{m-1}/R_{m-1}$, then we have $x_1,x_2,\ldots,x_{m+1}$
(instead of $[x_1],[x_2],\ldots,[x_{m+1}]$)
are fixed points of $X_{m}$ and $\lim_{k\rightarrow \infty}d(T^{m_k}y_k,x_{m+1})=0$.
Moreover, we have $y_k$ is minimal in $X_{m}.$

By Lemma \ref{semisimple}, we have $(x_1,x_2,\ldots,x_{m+1})\in SE_{m+1}(X_m,T)$
and so $SE_{m+1}^e(X,T)\neq\emptyset$ by repeatedly applying Proposition \ref{SE}.
Thus the proof is finished by induction.
\end{proof}

Now we are ready to show our statement.

\begin{thm} \label{NOT}
Let $(X,T)$ be a distal but not equicontinuous system,
then for every $n\geq 2,\ SE_n^e(X,T)\neq\emptyset.$
Consequently, the sequence entropy of $(X,T)$ is $\infty$
by the result in \cite{HYe}.
\end{thm}

\begin{proof}
Assume that
$(x,y)\in Q(X,T)\backslash \Delta_X$
with $x_k,y_k,z\in X$ and $n_k\in \Z,k=1,2,\ldots$
such that
$$\lim_{k\rightarrow \infty}(x_k,y_k)=(x,y)\quad
\text{and}
\quad \lim_{k\rightarrow \infty}(T\times T)^{n_k}(x_k,y_k)=(z,z).$$

As $(X,T)$ is distal,
it follows that for every $k\in \N,(x_k,y_k)$ is minimal in $X\times X$
and $\overline{Orb((x,y),T\times T)}\cap \Delta_X = \emptyset$. By collapsing $\overline{Orb((x,y),T\times T)}$ and $\Delta_X$
to points respectively, we obtain a factor $(Y,S)$ of $(X\times X, T\times T)$ which satisfies the
condition (\ref{condition}). This implies $SE_n^e(Y,S)\neq\emptyset$ for every $n\geq 2$.
Moreover, $SE_n^e(X\times X, T\times T)\neq\emptyset$ for every $n\geq 2$.

Fix $n\geq 2,$
choose $((u_i,v_i))_{i=1}^{n^2}\in SE_{n^2}^e(X\times X, T\times T)$.

For every $z\in X$,
let $I_z=\{u_i: 1\leq i\leq n^2, u_i=z\}$.
Clearly, there exists $1\leq m \leq n^2,z_j\in X,j=1,\ldots,m$
such that $I_{z_j}\neq \emptyset$.

If $m\geq n$, let $u_{i_j}\in I_{z_j}$ for some $1\leq i_j\leq n^2$,
then $((u_{i_j},v_{i_j}))_{j=1}^n\in  SE_{n}^e(X\times X, T\times T)$
by Remark \ref{AA}.
Now consider the projection $\pi_1$ from $X\times X$ to the first coordinate,
it follows that $(u_{i_j})_{j=1}^n\in  SE_{n}^e(X, T)$ by Proposition \ref{SE}.

If $m<n$, there exists $1\leq i_0\leq m$ with $l=|I_{z_{i_0}}|\geq\frac{n^2}{m}>n.$

Suppose that $u_{i_j}=z_{i_0},j=1,\ldots,l$,
then $((u_{i_j},v_{i_j}))_{j=1}^n\in  SE_{n}^e(X\times X, T\times T)$.
Now consider the projection $\pi_2$ from $X\times X$ to the second coordinate,
it follows that $(v_{i_j})_{j=1}^n\in  SE_{n}^e(X, T)$ by the Proposition \ref{SE}.
Thus we complete the proof.
\end{proof}

\section{Null systems and mean equicontinuity}

In this section we will discuss in which case a null dynamical system is mean equicontinuous.
We start with the following characterizations of mean equicontinuous systems.

Before showing our theorem, we need the following lemmas.

\begin{lem}\cite[Corollary 1]{A60} \label{P-is-closed}
If $(X,T)$ is a dynamical system and $P(X, T)$ is closed in $X\times X$, then it is a
$T\times T$-invariant closed equivalence relation.
\end{lem}

\begin{lem}\cite[Theorem 3.1]{NS}\label{NS}
If $(X,T)$
is a transitive not minimal $E$-system,
then we have $SE(X,T)\neq \emptyset$.
\end{lem}

Now we are ready to show

\begin{thm}\label{equivalence}
Let $(X,T)$ be a null system.
Then the following conditions are equivalent:

\begin{enumerate}
\item $P(X,T)$ is closed;
  \item $P(X,T)=Q(X,T);$
  \item $(X,T)$ is mean equicontinuious.
\end{enumerate}
\end{thm}

\begin{proof}
By \cite[Theorem 3.5]{LTY} we know that (3) implies (2). It is clear that (2) implies (1). It remains to show (1) $\Rightarrow$ (3).

Since $P(X,T)$ is closed in $X\times X$,  $P(X,T)$ is a
$T\times T$-invariant closed equivalence relation by Lemma ~\ref{P-is-closed}.
Let $Y=X/P(X,T)$ and $\pi:X\rightarrow Y$ be the factor map,
then $Y$ is the maximal distal factor.
As $(X,T)$ is a null system, so is $(Y,T)$ by Proposition \ref{SE}.
Moreover, $(Y,T)$ is equicontinuious by Theorem~\ref{mainthm}.

Assume the contrary $(X,T)$ is not mean equicontinuous, then there are points
$x_k,y_k,z\in X$, positive integers $n_k\in \N,k=1,2,\cdots$ and $\ep_0>0$
such that $\lim_{k\rightarrow\infty }x_k=z=\lim_{k\rightarrow\infty }y_k$,
and for every $k\in \N$, we have:
$$
\frac{1}{n_k}\sum_{i=0}^{n_k-1}d(T^ix_k,T^iy_k)\geq \ep_0.
$$

Let $\mu_k=\frac{1}{n_k}\sum_{i=0}^{n_k-1}\delta_{(T^ix_k,T^iy_k)}$.
We may assume $\mu_k\rightarrow \mu$ as $k\rightarrow \infty $
(otherwise consider the subsequence),
where $\mu\in M(X\times X,T\times T)$.

We claim that $\mu(\text{supp}(\mu )\setminus \Delta_X)>0$.
Actually, $d(\cdot,\cdot)$ is a continuous function on $X\times X$, then we have when $k\rightarrow \infty $,
$$
\int_{X\times X}d(x,y)\textrm{d}\mu_k\longrightarrow \int_{X\times X}d(x,y)\textrm{d}\mu
$$
and
$$
\int_{X\times X}d(x,y)\textrm{d}\mu_k=\frac{1}{n_k}\sum_{i=0}^{n_k-1}d(T^ix_k,T^iy_k)\geq \ep_0
$$
which implies
$$\mu(\text{supp}(\mu)\setminus \Delta_X)>0.$$

By ergodic decomposition, we have $\nu(\text{supp}(\mu)\setminus \Delta_X)>0$
for some ergodic measure $\nu$ on $X\times X $.

\medskip
\noindent \textbf{Case 1:} $\supp (\nu)$ is not minimal.

 It follows that $\supp(\nu) $
is  an $E$-system which is not minimal, then $SE(\supp(\nu),T\times T)\neq \emptyset$
by Lemma ~\ref{NS},
which implies $SE(X,T)\neq \emptyset.$

\medskip
\noindent \textbf{Case 2:} $\supp (\nu)$ is minimal.

Assume that $\supp (\nu)=\overline{Orb((u,v),T\times T)}$,
where $(u,v)\notin \Delta_X$.

For $l\in \N$,
let $B_l=\{(x,y)\in X\times X:d((x,y),(u,v))<\frac{1}{l} \}$,
then $\mu(B_l)>0$. Since $0<\mu(B_l)\leq \liminf_{k\rightarrow\infty}\mu_k(B_l)$
and
$$\mu_k(B_l)=\frac{1}{n_k}\#\{0\leq i \leq n_k-1:(T^ix_k,T^iy_k)\in B_l\},$$
thus there is $k_l, i_l \in \N$ with $0\leq i_l \leq n_{k_l}-1$ such that
$(T^{i_l}x_{k_l},T^{i_l}y_{k_l})\in B_l.$ Hence we have
\[
\lim_{l\rightarrow\infty}(x_{k_l},y_{k_l})=(z,z)
  \;\; \text{and}\;\;
  \lim_{l\rightarrow\infty}(T^{i_l}x_{k_l},T^{i_l}y_{k_l})=(u,v).
\]

Let $(Z,S)$ be the maximal equicontinuous factor of $(X,T)$
and  $\pi\colon (X,T)\to (Z,S)$ be the factor map.
Then $R_{\pi}=\{(x,y)\in X\times X: \pi(x)=\pi(y)\}\subset  P(X,T)$.
Fix $\ep>0$.
As $\pi $ is continuous, there is $\delta_1>0$ such that $d(\pi (a),\pi(b))<\frac{\ep}{3}$
whenever $a,b\in Z$ with $d(a,b)<\delta_1$.
As $(Z,S)$ is equicontinuous, there is $\delta_2>0$ such that $d(S^na,S^nb)<\min\{\frac{\ep}{3},\delta_1\}$
for every $n\in \N$
whenever $a,b\in Z$ with $d(a,b)<\delta_2$.
Put $\delta=\min\{\delta_1,\delta_2\}$.
Now choose $l\in \N$ such that $d(x_{k_l},y_{k_l})<\delta_2$ and
$d(u,T^{i_l}x_{k_l})<\delta_1,d(v,T^{i_l}y_{k_l})<\delta_1.$
Then $$d(\pi(T^{i_l}x_{k_l}),\pi(T^{i_l}y_{k_l}))=d(S^{i_l}\pi(x_{k_l}),S^{i_l}\pi(y_{k_l}) )<\frac{\ep}{3},$$
$$d(\pi(u),\pi(v))\leq d(\pi(u),\pi(T^{i_l}x_{k_l}))+d(\pi(T^{i_l}x_{k_l}),\pi(T^{i_l}y_{k_l}))+d((T^{i_l}y_{k_l}),\pi(v))<\ep$$
which implies $\pi(u)=\pi(v)$.
Moreover $(u,v)\in P(X,T)$,
$\overline{Orb((u,v),T\times T)}\cap \Delta_X \neq \emptyset$. It is a contradiction.
Thus $(X,T)$ is mean equicontinuous.
\end{proof}


\begin{thm}\label{semiss}
For a null system $(X,T)$, if it has dense minimal points, then it is mean equicontinuous.
\end{thm}

\begin{proof}
It is sufficient to show that $Q(X,T)=P(X,T)$ by Theorem \ref{equivalence}.

If there is a pair $(x,y)\in Q(X,T)\setminus P(X,T)$,
by the definition of $Q(X,T)$, there are
$x_k,y_k,z\in X$ and $n_k\in \N,k=1,2,\ldots$
such that
$$\lim_{k\rightarrow\infty }x_k=z=\lim_{k\rightarrow\infty }y_k
\quad \text{and} \quad   \lim_{k\rightarrow\infty}(T^{n_k}x_k,T^{n_k}y_k)=(x,y).$$
As the minimal points for $T$ are dense in $X$,
so are for $T\times T$ in $X\times X$ (see \cite{AG01}).
Without loss of generality, we can assume that $(x_k,y_k)$ is minimal
in $X\times X$ for every $k\in \N.$

Put $\mathbf{x}=(x,y),\mathbf{x_k}=(x_k,y_k),\mathbf{z}=(z,z)$ and $\widetilde{X}=X\times X,\widetilde{T}=T\times T.$

Since $(x,y)\notin P(X,T)$, we have $\overline{Orb(\mathbf{x},\widetilde{T})}\cap\overline{Orb(\mathbf{z},\widetilde{T})}=\emptyset$.
Let
\[
R=\{(\mathbf{a},\mathbf{b})\in \widetilde{X} \times \widetilde{X}:
\mathbf{a},\mathbf{b}\in \overline{Orb(\mathbf{x},\widetilde{T})}
\;\;\text{or}\;\;  \overline{Orb(\mathbf{z},\widetilde{T})}\}\cup\Delta_{\widetilde{X}}.
\]

It is easy to see that $R$ is a closed $\widetilde{T}\times \widetilde{T}$-invariant
equivalence relation on $\widetilde{X}$.
Let $Y=\widetilde{X}/ R$, it follows that $[\mathbf{x}],[\mathbf{z}]$ are
fixed points in $Y$ and
$\lim_{k\rightarrow\infty}[\mathbf{x_k}]=[\mathbf{z}], \lim_{k\rightarrow\infty}\widetilde{T}^{n_k}[\mathbf{x_k}]=[\mathbf{x}].$
As $\mathbf{x_k}$ is minimal in $\widetilde{X}$, so is $[\mathbf{x_k}]$ in $Y$.
By Lemma \ref{semisimple}, we have $([\mathbf{x}],[\mathbf{z}])\in SE(\widetilde{X},\widetilde{T}),$
which implies $SE(X,T)\neq \emptyset$  by Proposition \ref{SE}.

It is a contradiction which shows $Q(X,T)=P(X,T)$.
\end{proof}

\begin{cor}
Let $(X,T)$ be a null system.
If $\supp(X,T)=X$, then it is mean equicontinuous.
\end{cor}
\begin{proof}
By \cite[Lemma 3.3]{LT14}, we have
\[
\supp(X,T)=\overline{\bigcup_{\mu\in M(X,T)}\supp(\mu)}
=\overline{\bigcup_{\nu\in M^e(X,T)}\supp(\nu)}.
\]

Clearly, $(\supp(\nu),T)$ is minimal since a transitive and non-minimal $E$-system is not null \cite[Theorem 3.1]{NS}.
Thus the minimal points for $T$ in $X$ are dense
which implies $(X,T)$ is mean equicontinuous
by Theorem~\ref{semiss}.
\end{proof}





\section{A counterexample}
In the last part of this paper, we give the example which is mentioned in the introduction. That is,
\begin{exmp}\label{counter}
There is a distal system $(X,T)$ (consisting of periodic orbits)
such that $\mathbf{RP}^{[\infty]}(X, T)\neq \Delta_X$
and $\text{Ind}_{fip}(X, T)=\emptyset.$
\end{exmp}

Let $A\in[0,1]\times [-1,1]$, denote the horizontal ordinate and vertical ordinate by $x_A$ and $y_A$ respectively.
We define a relation $\sim$ on $[0,1]\times [-1,1]$ by
$(x,1)\sim (x,-1)$ for $x\in [0,1]$.
Clearly, it is an equivalence relation.

Assume $A(x_1,y_1),B(x_2,y_2)\in [0,1]\times [-1,1]/\sim$,
let $$\rho(A,B)=\sqrt{|x_1-x_2|^2+(\min\{|y_1-y_2|,2-|y_1-y_2|\})^2},$$
it is easy to check that $\rho$ is a metric on $[0,1]\times [-1,1]/\sim$.

\bigskip

We begin to construct $(X,T)$ in $[0,1]\times [-1,1]/\sim$.

\noindent {\bf Step 1: Basic periodic systems.}

First,
we construct basic periodic systems $(I_i,T_i)$. Put
$$
I_i=\{(0,\frac{k}{2i}):k=0,\cdots,2i\}
\cup \{(\frac{s}{2^i},\frac{j}{2i}):s=1,\cdots,2^i;j=0,1\}
\cup \{(1,\frac{-t}{2i}):t=1,\cdots,2i\}.
$$
The map $T_i$ from $I_i$ to itself is defined by
$$
T_iA=
  \begin{cases}
    (0,\frac{k-1}{2i})&  \text{if $A\in \{(0,\frac{k}{2i}):k=2,3,\cdots,2i\}$},\\
   (\frac{s+1}{2^i},\frac{j}{2^i})&  \text{if $A\in \cup \{(\frac{s}{2^i},\frac{j}{2i}):s=0,1,\cdots,2^i-1;j=0,1\}$ and $s+j$ is even},\\
     (\frac{s}{2^i},\frac{j+1\mod2}{2i})&   \text{if $A\in \cup \{(\frac{s}{2^i},\frac{j}{2i}):s=0,1,\cdots,2^i;j=0,1\}$ and $s+j$ is odd},\\
       (1,-\frac{t+1}{2i})&  \text{if $A\in \{(1,\frac{-t}{2i}):k=0,1,\cdots,2i-1\}$},\\
       (0,1)& \text{if $A=(1,-1)$}  .
  \end{cases}
$$
such that $(I_i,T_i)$ is a periodic system.

 See the following figures for the construction.

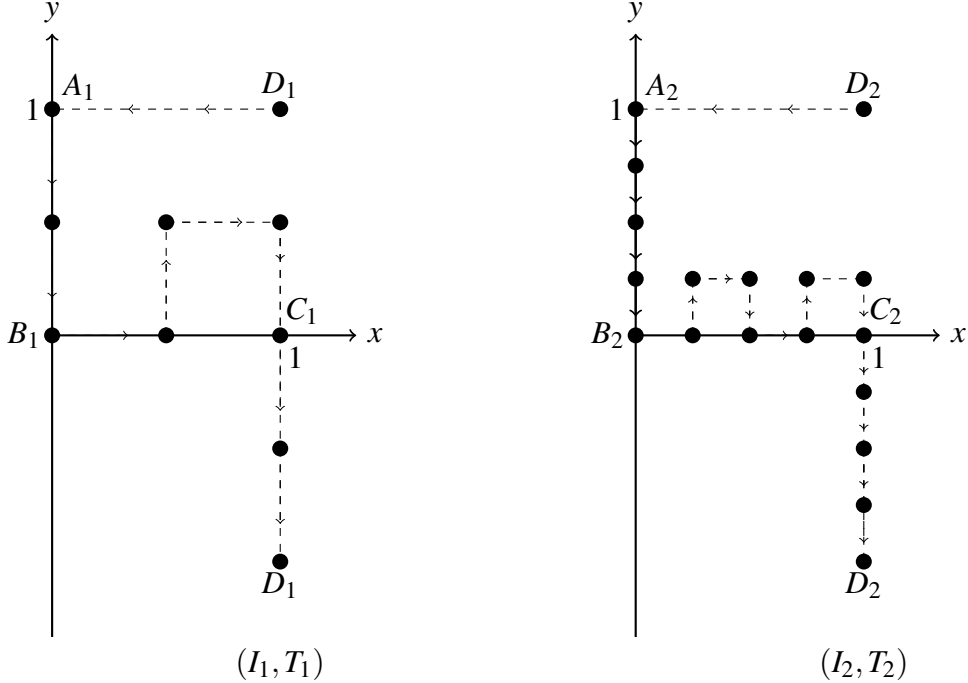
\begin{multicols}{2}
\begin{tikzpicture}
\draw [thick,<->] (0,4) node[above]{$y$}--(0,0)--(4,0)node[right]{$x$};
\draw  [thick,-]  (0,0)--(0,-4);
\draw  [dashed,->] (0,3)--(0,2);
\draw  [dashed,->]  (0,2)--(0,0.5);
\draw  [dashed,->]  (0,0)--(1,0);
\draw  [dashed,->]  (1.5,0)--(1.5,1.5);
\draw  [dashed,->]  (1.5,0)--(1.5,1);
\draw  [dashed,-]  (1.5,1.5)--(3,1.5);
\draw  [dashed,->]  (1.5,1.5)--(2.5,1.5);
\draw  [dashed,-]  (3,1.5)--(3,0);
\draw  [dashed,->]  (3,1.5)--(3,1);
\draw  [dashed,-]  (3,0)--(3,-1.5);
\draw  [dashed,->]  (3,1.5)--(3,-1);
\draw  [dashed,-]  (3,-1.5)--(3,-3);
\draw  [dashed,->]  (3,-1.5)--(3,-2.5);
\draw  [dashed,->]  (3,3)--(2,3);
\draw  [dashed,->]  (2,3)--(1,3);
\draw  [dashed,->]  (1,3)--(0,3);
\node [above] at (3.3,0) {$C_1$};
\node [below] at (3.2,0) {1};
\node [left] at (0,0) {$B_1$};
\node [right] at (0,3.3) {$A_1$};
\node [left] at (0,3) {$1$};
\node [below] at (3,-3) {$D_1$};
\node [above] at (3,3) {$D_1$};
\fill (0,0)circle (3pt) (3,3)circle (3pt)(0,1.5)circle (3pt) (0,3)circle (3pt)
(1.5,0)circle (3pt)  (1.5,1.5)circle (3pt) (3,0)circle (3pt) (3,1.5)circle (3pt)(3,-1.5)circle (3pt)
(3,-3)circle (3pt);
\node [below] at (3,-4) {$(I_1,T_1)$};
\end{tikzpicture}

\begin{tikzpicture}
\draw [thick,<->] (0,4) node[above]{$y$}--(0,0)--(4,0)node[right]{$x$};
\draw  [thick,-]  (0,0)--(0,-4);
\draw  [thick,->]  (0,3)--(0,2.5);
\draw  [thick,->]  (0,2.25)--(0,1.75);
\draw  [thick,->]  (0,1.5)--(0,1);
\draw  [thick,->]  (0,0.75)--(0,0.25);
\draw  [dashed,-]  (0.75,0)--(0.75,0.75);
\draw  [dashed,->]  (0.75,0)--(0.75,0.5);
\draw  [dashed,-]  (0.75,0.75)--(1.5,0.75);
\draw  [dashed,->]  (0.75,0.75)--(1.25,0.75);
\draw  [dashed,-]  (1.5,0.75)--(1.5,0);
\draw  [dashed,->]  (1.5,0.75)--(1.5,0.25);
\draw  [dashed,-]  (1.5,0)--(2.25,0);
\draw  [dashed,->]  (1.5,0)--(2,0);
\draw  [dashed,-]  (2.25,0)--(2.25,0.75);
\draw  [dashed,->]  (2.25,0)--(2.25,0.5);
\draw  [dashed,-]  (2.25,0.75)--(3,0.75);
\draw  [dashed,->]  (3,0.75)--(3,0.25);
\draw  [dashed,-]  (3,0)--(3,-0.75);
\draw  [dashed,->]  (3,0)--(3,-0.5);
\draw  [dashed,-]  (3,-0.75)--(3,-1.5);
\draw  [dashed,->]  (3,-0.75)--(3,-1.25);
\draw  [dashed,-]  (3,-1.5)--(3,-2.25);
\draw  [dashed,->]  (3,-1.5)--(3,-2);
\draw  [dashed,-]  (3,-2.25)--(3,-3);
\draw  [dashed,->]  (3,-1.5)--(3,-2.75);
\draw  [dashed,->]  (3,3)--(2,3);
\draw  [dashed,->]  (2,3)--(1,3);
\draw  [dashed,->]  (1,3)--(0,3);
\node [above] at (3.3,0) {$C_2$};
\node [below] at (3.2,0) {$1$};
\node [left] at (0,0) {$B_2$};
\node [left] at (0,3) {$1$};
\node [right] at (0,3.3) {$A_2$};
\node [below] at (3,-3) {$D_2$};
\node [above] at (3,3) {$D_2$};
\fill (0,0)circle (3pt) (3,3)circle (3pt)(0,0.75)circle (3pt) (0,1.5)circle (3pt) (0,2.25)circle (3pt)(0,3)circle (3pt)
(0.75,0)circle (3pt) (0.75,0.75)circle (3pt)
(1.5,0)circle (3pt) (1.5,0.75)circle (3pt)
(2.25,0)circle (3pt) (2.25,0.75)circle (3pt)
(3,0.75)circle (3pt) (3,0)circle (3pt) (3,-0.75)circle (3pt)
 (3,-1.5)circle (3pt)  (3,-2.25)circle (3pt)  (3,-3)circle (3pt);
 \node [below] at (3,-4) {$(I_2,T_2)$};
\end{tikzpicture}
 \end{multicols}

\begin{center}

\begin{tikzpicture}
\draw [thick,<->] (0,8.5) node[above]{$y$}--(0,0)--(8.5,0)node[right]{$x$};
\draw  [thick,-]  (0,0)--(0,-8.5);
\node [left] at (0,8) {$1$};
\node [left] at (0,0) {$B_3$};
\node [above] at (8,8) {$D_3$};
\node [right] at (8.2,-8) {$D_3$};
\node [right] at (0,8.3) {$A_3$};
\node [above] at (8.4,0) {$C_3$};
\node [below] at (8.2,0) {$1$};
\node [below] at (8,-8.4) {$(I_3,T_3)$};
\fill (0,0)circle (3pt)  (0,1)circle (3pt)  (0,2)circle (3pt) (0,3)circle (3pt) (0,4)circle (3pt)
 (0,5)circle (3pt) (0,6)circle (3pt) (0,7)circle (3pt) (0,8)circle (3pt)
 (1,0)circle (3pt) (2,0)circle (3pt) (3,0)circle (3pt) (4,0)circle (3pt) (5,0)circle (3pt) (6,0)circle (3pt) (7,0)circle (3pt)
  (8,0)circle (3pt)
   (1,1)circle (3pt) (2,1)circle (3pt) (3,1)circle (3pt) (4,1)circle (3pt) (5,1)circle (3pt) (6,1)circle (3pt) (7,1)circle (3pt)
  (8,1)circle (3pt)
  (8,-1)circle (3pt)(8,-2)circle (3pt)(8,-3)circle (3pt)(8,-4)circle (3pt)(8,-5)circle (3pt)
  (8,-6)circle (3pt)(8,-7)circle (3pt)(8,-8)circle (3pt)(8,8)circle (3pt) ;
 \draw  [dashed,->]  (0,8)--(0,7.5);  \draw  [dashed,->]  (0,2)--(0,1.5);
  \draw  [dashed,->]  (0,7)--(0,6.5);   \draw  [dashed,->]  (0,1)--(0,0.5);
   \draw  [dashed,->]  (0,6)--(0,5.5);    \draw  [dashed,->]  (0,0)--(0.5,0);
    \draw  [dashed,->]  (0,5)--(0,4.5);   \draw  [dashed,->]  (1,0)--(1,0.5);
     \draw  [dashed,->]  (0,4)--(0,3.5);    \draw  [dashed,->]  (1,0.5)--(1,1);
      \draw  [dashed,->]  (0,3)--(0,2.5);    \draw  [dashed,->]  (1,1)--(1.5,1);
       \draw  [dashed,->]  (1.5,1)--(2,1);   \draw  [dashed,->]  (2,1)--(2,0.5);
       \draw  [dashed,->]  (2,0.5)--(2,0);
       \draw  [dashed,->]  (2,0)--(2.5,0);
           \draw  [dashed,->]  (3,0)--(3,0.5); \draw  [dashed,->]  (3,0.5)--(3,1);
            \draw  [dashed,->]  (3,1)--(3.5,1); \draw  [dashed,->]  (3.5,1)--(4,1);
             \draw  [dashed,->]  (4,1)--(4,0.5);  \draw  [dashed,->]  (4,0.5)--(4,0);
              \draw  [dashed,->]  (4,0)--(4.5,0);  \draw  [dashed,->]  (5,0)--(5,0.5);
               \draw  [dashed,->]  (5,0.5)--(5,1);    \draw  [dashed,->]  (5,1)--(5.5,1);
                 \draw  [dashed,->]  (5.5,1)--(6,1);   \draw  [dashed,->]  (6,1)--(6,0.5);
                 \draw  [dashed,->]  (6,0.5)--(6,0);\draw  [dashed,->]  (6,0)--(6.5,0);
                 \draw  [dashed,->]  (7,0)--(7,0.5);    \draw  [dashed,->]  (7,0.5)--(7,1);
                     \draw  [dashed,->]  (7,1)--(7.5,1);  \draw  [dashed,->]  (7.5,1)--(8,1);
                     \draw  [dashed,->]  (8,1)--(8,0.5);\draw  [dashed,->]  (8,0.5)--(8,0);
                     \draw  [dashed,->]  (8,0)--(8,-0.5);\draw  [dashed,->]  (8,-0.5)--(8,-1);
                     \draw  [dashed,->]  (8,-1)--(8,-1.5);\draw  [dashed,->]  (8,-1.5)--(8,-2);
                     \draw  [dashed,->]  (8,-2)--(8,-2.5);\draw  [dashed,->]  (8,-2.5)--(8,-3);
                     \draw  [dashed,->]  (8,-3)--(8,-3.5);\draw  [dashed,->]  (8,-3.5)--(8,-4);
                     \draw  [dashed,->]  (8,-4)--(8,-4.5);\draw  [dashed,->]  (8,-4.5)--(8,-5);
                     \draw  [dashed,->]  (8,-5)--(8,-5.5);\draw  [dashed,->]  (8,-5.5)--(8,-6);
                     \draw  [dashed,->]  (8,-6)--(8,-6.5);\draw  [dashed,->]  (8,-6.5)--(8,-7);
                     \draw  [dashed,->]  (8,-7)--(8,-7.5);\draw  [dashed,->]  (8,-7.5)--(8,-8);
                        \draw  [dashed,->]  (8,8)--(7,8);
                        \draw  [dashed,->]  (7,8)--(6,8);
                        \draw  [dashed,->]  (6,8)--(5,8);
                        \draw  [dashed,->]  (5,8)--(4,8);
                        \draw  [dashed,->]  (4,8)--(3,8);
                        \draw  [dashed,->]  (3,8)--(2,8);
                        \draw  [dashed,->]  (2,8)--(1,8);
                        \draw  [dashed,->]  (1,8)--(0,8);
\end{tikzpicture}
\end{center}

 For every $i$, put $A_i=(0,1),B_i=(0,0),C_i=(1,0),D_i=(1,-1)$,
and from the construction, it is easy to see that
 $T^{2^{i+1}}_iB_i=C_i,T_i^{2i}A_i=B_i$.

 \bigskip

\noindent {\bf Step 2: Construction of $(X,T)$.}

To get the properties, we need glue all the basic periodic systems constructed above.
For compactness, it is necessary to compress them first.
Precisely, for every $i$,
squeeze $I_i $ horizontally such that
the distance from $B_i$ to $C_i$ is $\frac{1}{2i(i+1)}$,
and horizontally translate
$B_i$ to $B'_i(\frac{2i+1}{2i(i+1)},0)$
and $C_i$ to $C_i'(\frac{1}{i},0)$
which is denoted by $I'_i$.
The map $T_i'$ defined on $I'_i$ can be induced from $T_i$,
that is, $T_i'$ keeps the relative position.

Let $X=\overline{\bigcup_{i=1}^\infty I'_i}$ and $T$ be the map
from $X$ to itself such that
$T\mid_{I'_i}=T'_i$  and $T\mid_{\{0\}\times [-1,1]}$ is identity.

Put $A=\lim_{i\rightarrow\infty}A'_i(=\lim_{i\rightarrow\infty}D'_i)$
, $B=\lim_{i\rightarrow\infty}B'_i(=\lim_{i\rightarrow\infty}C'_i)$
and $I=\{0\}\times [-1,1]$
where $A'_i(\frac{2i+1}{2i(i+1)},1),D'_i(\frac{1}{i},-1)$.

See the following figure for the construction.

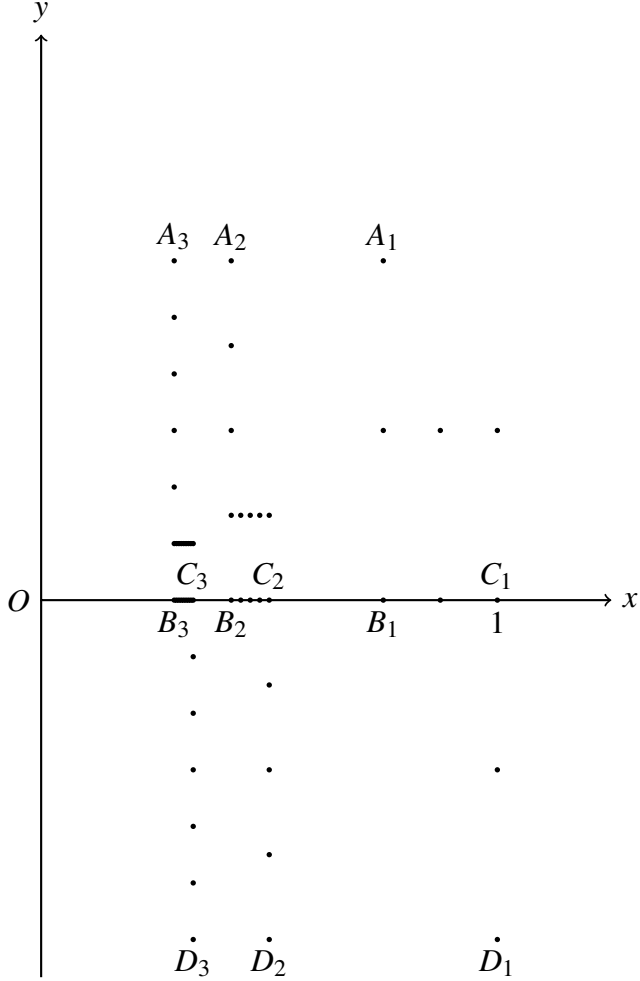
\bigskip

\begin{center}

\begin{tikzpicture}
\draw [thick,<->] (0,7.5) node[above]{$y$}--(0,0)--(7.5,0)node[right]{$x$};
\draw  [thick,-]  (0,0)--(0,-5);
\node [above] at (6,0) {$C'_1$};
\node [below] at (6,0) {$1$};
\node [below] at (6,-4.5) {$D'_1$};
\node [below] at (4.5,0) {$B'_1$};
\node [above] at (4.5,4.5) {$A'_1$};
\node [above] at (3,0) {$C'_2$};
\node [below] at (3,-4.5) {$D'_2$};
\node [below] at (2.5,0) {$B'_2$};
\node [above] at (2.5,4.5) {$A'_2$};
\node [above] at (2,0) {$C'_3$};
\node [below] at (2,-4.5) {$D'_3$};
\node [below] at (1.75,0) {$B'_3$};
\node [above] at (1.75,4.5) {$A'_3$};
\fill (6,2.25)circle (1pt)  (6,0) circle (1pt)  (6,-2.25) circle (1pt)  (6,-4.5) circle (1pt)
        (5.25,2.25)circle (1pt)  (5.25,0) circle (1pt)
  (4.5,0)circle (1pt)  (4.5,2.25) circle (1pt)  (4.5,4.5) circle (1pt)
  (3,1.125)circle (1pt)  (3,0) circle (1pt)  (3,-1.125) circle (1pt)
  (3,-2.25) circle (1pt) (3,-3.375) circle (1pt) (3,-4.5) circle (1pt)

  (2.875,1.125)circle (1pt)  (2.875,0) circle (1pt)
  (2.75,1.125)circle (1pt)  (2.75,0) circle (1pt)
  (2.625,1.125)circle (1pt)  (2.625,0) circle (1pt)
    (2.5,0)circle (1pt)  (2.5,1.125) circle (1pt)  (2.5,2.25) circle (1pt) (2.5,3.375) circle (1pt) (2.5,4.5) circle (1pt)
  (2,0.75)circle (1pt) (2,0)circle (1pt)  (2,-0.75)circle (1pt)  (2,-1.5)circle (1pt)
   (2,-2.25)circle (1pt)  (2,-3)circle (1pt)  (2,-3.75)circle (1pt) (2,-4.5)circle (1pt)

(1.75,0.75)circle (1pt) (1.75,0)circle (1pt)   (1.75,1.5)circle (1pt)
 (1.75,2.25)circle (1pt)  (1.75,3)circle (1pt)  (1.75,3.75)circle (1pt) (1.75,4.5)circle (1pt)

 (1.78125,0)circle (1pt)    (1.78125,0.75)circle (1pt)

   (1.8125,0)circle (1pt)    (1.8125,0.75)circle (1pt)

   (1.84375,0)circle (1pt)    (1.84375,0.75)circle (1pt)
 (1.875,0)circle (1pt)       (1.875,0.75)circle (1pt)

     (1.90625,0)circle (1pt)    (1.90625,0.75)circle (1pt)
(1.9375,0)circle (1pt)    (1.9375,0.75)circle (1pt)
     (1.96875,0)circle (1pt)    (1.96875,0.75)circle (1pt);
\end{tikzpicture}
\end{center}

\bigskip

\noindent {\bf Claim 1}: For each $d\geq 1$, we have
$(A,B)\in \mathbf{RP}^{[d]}(X,T)$, hence $(A,B)\in \mathbf{RP}^{[\infty]}(X,T).$

\begin{proof}
Fix $d\geq 1$, for every $\varepsilon>0$, there is an $i\in \mathbb{N}$ with
$\rho(B'_i,B)=\rho(A'_i,A)=\frac{2i+1}{2i(i+1)}<\varepsilon$
and $(1+d)di<2^i$.

Put $\bar{n}=(2i,\ldots,2di)\in \mathbb{Z}_+^d$.
For each $\alpha \in \{0,1\}^d\backslash  \{\mathbf{0}\}$,
we have
$2i\leq \bar{n}\cdot\alpha\leq (1+d)di<2^{i}$.
It follows the construction that
$T^{2^{i+1}}B'_i=C'_i,T^{2i}A'_i=B'_i$,
which implies
\[
\frac{2i+1}{2i(i+1)}\leq x_{T^{\bar{n}\cdot\alpha}A'_i}, x_{T^{\bar{n}\cdot\alpha}B'_i}
\leq \frac{1}{i} \
\text{and}\
 0\leq y_{T^{\bar{n}\cdot\alpha}A'_i},y_{T^{\bar{n}\cdot\alpha}B'_i}\leq \frac{1}{2i}.
\]
Thus
$$
\rho(T^{\bar{n}\cdot\alpha}A'_i,T^{\bar{n}\cdot\alpha}B'_i)<
\sqrt{(\frac{1}{2i})^2+(\frac{1}{2i(i+1)})^2}<\varepsilon
$$
which means $(A,B)\in \mathbf{RP}^{[d]}(X,T)$,
hence $(A,B)\in \mathbf{RP}^{[\infty]}(X,T).$
\end{proof}

\begin{rem}
Actually, by the similar argument, we can show that
for every $C\in I$ with $C\neq B$, $(C,B)\in \mathbf{RP}^{[\infty]}(X,T).$
\end{rem}

\bigskip

\noindent {\bf Claim 2}:
$\text{Ind}_{fip}(X,T)=\emptyset.$

\begin{proof}
If there is a pair $(C,D)\in \text{Ind}_{fip}(X,T)$, then $C,D\in I.$
Indeed, as $\text{Ind}_{fip}(I_i',T)=\emptyset$ for every $i\in \N$,
we obtain that $C,D$ cannot belong to the same $I_i'$.
Now if $C$ belongs to $I_k'$ for some $k\in\N$,
we can choose the neighborhoods $U,V$ of $C,D$
respectively which are small enough such that
$(\bigcup_{n\in\N}T^{-n}U ) \cap  (\bigcup_{n\in\N}T^{-n}V )=\emptyset$.
Therefore $(C,D)\notin \text{Ind}_{fip}(X,T)$,
contracting the assumption.

Without loss of generality,
assume that $C=(0,c),D=(0,d)$ with $0\leq d < c\leq 1.$

Choose $k\in \N$ such that $c>\frac{2}{k}$.
Let $U_1,U_2$ be $C,D$ neighbourhoods respectively with
$\diam(\overline{U_1})<\min\{\frac{1}{4}(c-d),\frac{2k+1}{2k(k+1)}\}$ and
$\rho(\overline{U_1},\overline{U_2})>\frac{3}{4}(c-d)$.

As $(C,D)\in \text{Ind}_{fip}(X,T)$,
there are $n_1,n_2\in \Z$ and $P\in X$
with
\[
P\in U_1\cap T^{-n_1}U_1\cap T^{-n_2}U_1  \cap T^{-(n_1+n_2)}U_2.
\]

Assume that $P\in I'_i$  and the period of $P$ is $m_i$.
As $P\in U_1$ and $\diam(\overline{U_1})<\frac{2k+1}{2k(k+1)}$,
we obtain $x_P<\frac{2k+1}{2k(k+1)}$
which implies $i\geq k$ and
\[
\diam(\overline{U_1})\geq \sqrt{(x_P-x_C)^2+(y_P-y_C)^2}\geq x_P\geq\frac{2i+1}{2i(i+1)}>\frac{1}{2i}.
\]

For $j=1,2$,
we have $n_j=k_jm_i+r_j$ where $k_j\in \Z$ and
$0\leq r_j<m_i$ with
$$\rho(T^{n_j}P,P)=\rho(T^{r_j}P,P)<\diam(\overline{U_1}).$$

As $\frac{c}{4}>\diam(\overline{U_1})\geq d(C,P)>|c-y_P|$,
we obtain $y_P>\frac{3}{4}c>\frac{1}{k}>\frac{1}{2i}$.
Similarly, we also have $y_{T^{r_2}P}>\frac{1}{2i}$.
It follows the definition of $I_i$
that the maximal distance of two adjacent points is $\frac{1}{2i}$,
that is, $\rho(TQ,Q)\leq\frac{1}{2i}$, for every $Q\in I_i$.
We have that $\rho(T^{r_2}P,P)+\frac{1}{2i}\geq \rho(T^{r_2+s}P,T^{s}P)$ for any $s\in \N$.

Hence
\begin{align*}
\rho(T^{n_1+n_2}P,P)=\rho(T^{r_1+r_2}P,P)\leq
&\rho(T^{r_1+r_2}P,T^{r_1}P)+\rho(T^{r_1}P,P)\\
&\leq\rho(T^{r_2}P,P)+\frac{1}{2i}+\rho(T^{r_1}P,P)\\
&<3\diam(\overline{U_1})<\frac{3}{4}(c-d).
\end{align*}

As $\rho(\overline{U_1},\overline{U_2})>\frac{3}{4}(c-d)$,
it follows that $T^{n_1+n_2}P\notin U_2.$
It is a contradiction which implies the claim.
\end{proof}

\section{Some open questions}
Some questions concerning null systems are open. Here we state two of them.
The first one appears in \cite{NS}.
\begin{ques}\label{q-1} Is a null transitive system minimal?
\end{ques}

We conjecture that Question \ref{q-1} has a negative answer. A question we can not solve in this paper is:
\begin{ques} Is a null transitive system mean equicontinuous?
\end{ques}


\medskip

\bibliographystyle{amsplain}

\end{document}